\documentclass[12pt,twoside]{amsart}
\usepackage{amsmath, amsthm, amscd, amsfonts, amssymb, graphicx}
\newtheorem{thm}{Theorem}[section]

\newtheorem{prop}[thm]{Proposition}
\newtheorem{defn}[thm]{Definition}

\newtheorem{problem}[thm]{Example}
\numberwithin{equation}{section}

\begin{document}


\begin{center}\large{{\bf{Fuzzy soft numbers}}}

\vspace{0.5cm}

\footnotesize Manash Jyoti Borah$^{1}$ and Bipan Hazarika$^{2}$

\vspace{.2cm}
\footnotesize $^{1}$Department of Mathematics, Bahona College,  Jorhat-785 101, Assam, India\\
Email:mjyotibora9@gmail.com

\vspace{.2cm}

\footnotesize $^{2}$Department of Mathematics, Gauhati University, Guwahati-781014, Assam, India\\
Email:  bh\_gu@gauhati.ac.in;  bh\_rgu@yahoo.co.in\\



\end{center}
\title{}
\author{}
\thanks{{\today\\$^{\ast}$The corresponding author}}

\begin{abstract}

 In this paper, we introduce the notion of fuzzy soft numbers. Here defined fuzzy soft number and four arithmetic operations $ \tilde{+}, \tilde{-}, \tilde{\times}, \tilde{\div} $ and related properties. Also introduce Hausdorff distance, Fuzzy soft metric space, convergence sequence, Cauchy sequence, Continuity, and uniform continuity of fuzzy soft numbers. At starting of this paper, we study convex and concave fuzzy soft sets and
 some of their properties. \\


AMS subject classification no: 03E72.
\end{abstract}

\maketitle
\pagestyle{myheadings}
%

\maketitle
\section{Introduction}
In 2003, Maji et al. \cite{majietal1,majietal2,maji} studied the theory of soft set initiated by Molodtsov    \cite{molodstov} and developed several basic notions of soft set theory. They also applied soft set theory in decision making (see \cite{cagman}) and solving problems in medical, economics, engineering, etc.

In the year 1965, Zadeh \cite{zadeh} introduced the concept of fuzzy set theory and its applications can be found in
many branches of mathematical and engineering sciences including management science, control
engineering, computer science, artificial intelligence. In 1970, firstly introduced concept of fuzzy number. The theory of fuzzy number has a number of applications such as fuzzy topology, fuzzy analysis, fuzzy logic and fuzzy decision making, algebraic structures,  etc. But fuzzy soft number are not developed several directions. In 2012, Das et al. \cite{das1,das2} introduced important part of soft set theory, which are notion of soft real sets, soft real numbers, soft complex numbers and some of their basic properties. Also Das et al. \cite{das3} introduced a notion of soft metric and some basic properties of soft metric space.

In this paper, we study notion of fuzzy soft number. In Section 2, some preliminary results are given. In Section 3, study Convex and concave fuzzy soft sets and some of their properties. After this, we study the notion of fuzzy soft numbers and some of basic properties. In Section 4, introduce Hausdorff distance, Fuzzy soft metric space, convergence sequence, Cauchy sequence, Continuity and uniform continuity of fuzzy soft numbers.

\section{Preliminary Results}

In this section we recall some basic concepts and definitions
regarding fuzzy soft sets, fuzzy soft topology and fuzzy soft mapping.

\begin{defn} \cite{maji} Let $U$ be an initial universe and $F$ be a set of parameters. Let $\tilde{P}(U)$ denote the power set of $U$ and $A$ be a non-empty subset of $F.$  Then $F_A$ is called a fuzzy soft set over U where $F:A\rightarrow \tilde{P}(U)$is a mapping from $A$ into $\tilde{P}(U).$
\end{defn}

\begin{defn}
\cite{molodstov}  $F_E$ is called a soft set over $U$ if and only if $F$ is a mapping of $E$ into the set of all
subsets of the set $U.$
\end{defn}
In other words, the soft set is a parameterized family of subsets
of the set $U.$ Every set $F(\epsilon),$ $\epsilon\tilde{\in} E,$
from this family may be considered as the set of
$\epsilon$-element of the soft set $F_E$  or as the set of
$\epsilon$-approximate elements of the soft set.

\begin{defn}
	\cite{kharal} Let $X$ be a universe and $E$ a set of attributes. Then the collection of all fuzzy soft sets over $U$ with attributes from $E$ is called a fuzzy soft class and is denoted by $\overline{(X,E)}.$
\end{defn}
\begin{defn}
	\cite{borah}	Let $ \overline{(U,E)} $ and $ \overline{(V,E')} $ be classes of hesitant fuzzy soft sets over $ U $ and $ V $ with attributes from $ E $ and $ E' $
	respectively. Let $ p:U\longrightarrow V $ and  $ q:E\longrightarrow E' $ be mappings. Then a hesitant fuzzy soft mappings $ f=(p,q):\overline{(U,E)}\longrightarrow\overline{(V,E')} $ is defined as follows; \\
	For a hesitant fuzzy soft set $F_A$ in $\overline{(U,E)},$ $ f(F_A) $ is a hesitant fuzzy soft set in $ \overline{(V,E')} $ obtained as follows: for $ \beta\tilde{\in}q(E)\tilde{\subseteq}E'$ and $ y\tilde{\in}V, $
	\[f(F_A)(\beta)(y)=\bigcup_{\alpha\tilde{\in}q^-1(\beta)\cap A,s\tilde{\in}p^-1(y)} (\alpha)\mu_s\]
	$ f(F_A)$  is called a hesitant fuzzy soft image of a hesitant fuzzy soft set  $F_A.$ Hence $ (F_A, f(F_A))\tilde{\in}f, $ where $ F_A\tilde{\subseteq}\overline{(U,E)}, f(F_A)\tilde{\subseteq}\overline{(V,E')}. $
\end{defn}
\begin{defn}
	\cite{borah}Let  $ f=(p,q):\overline{(U,E)}\longrightarrow\overline{(V,E')} $ be a hesitant fuzzy soft mapping and $ G_B, $ a hesitant fuzzy soft set in $\overline{(V,E')} $, where $ p:U\longrightarrow V ,  q:E\longrightarrow E' $ and $ B\tilde{\subseteq}E'. $ Then $ f^{-1}(G_B) $ is a hesitant fuzzy soft set in $\overline{(U,E)} $ defined as follows: for $ \alpha\tilde{\in}q^{-1}(B)\tilde{\subseteq}E $ and $ x\tilde{\in}U, $
	\[f^{-1}(G_B)(\alpha)(x)=(q(\alpha))\mu_{p(x)}  \]
	$f^{-1}(G_B)  $ is called a hesitant fuzzy soft inverse image of $ G_B. $
\end{defn}
\begin{defn} \cite{roy} A fuzzy soft topology $\tau$ on
	$(U,E)$ is a family of fuzzy soft sets over $(U,E)$ satisfying the
	following properties
	\begin{enumerate}
		\item[(i)] $\tilde{\phi},\tilde{E}\tilde{\in}\tau$
		\item[(ii)] if $F_{A},G_{B}\tilde{\in}\tau,$  then
		$F_{A}\tilde{\cap}G_{B}\tilde{\in}\tau.$
		\item[(iii)] if ${F_{A_{\alpha}}\tilde{\in}}\tau$ for all
		$\alpha\tilde{\in}\Delta$ an index set, then
		$\bigcup_{\alpha\in{\Delta}}F_{A_{\alpha}}\tilde{\in}\tau.$
	\end{enumerate}
\end{defn}
\begin{defn} \cite{roy} If $\tau$ is a fuzzy soft topology on $ (U,E), $ the triple $(U,E, \tau)$ is said to be a fuzzy soft topological space. Also each member of $ \tau $ is called a fuzzy soft open set in $(U,E, \tau).$
\end{defn}
\begin{defn} \cite{neog} Let $(U,E,\tau)$ be a fuzzy soft
	topological space. Let $F_{A}$ be a fuzzy soft set over $(U,E).$ The
	fuzzy soft closure of $F_A$ is defined as the intersection of all
	fuzzy soft closed sets which contained $F_A$ and is denoted by
	$\bar{F_A}$ or $ cl(F_A) $ we write
		\[cl(F_A)=\tilde{\bigcap}\{G_B: G_B\textrm {~is fuzzy soft closed
		and~} F_A\tilde{\subseteq}G_B\}.\]
\end{defn}
\begin{defn} \cite{aygunoglu}	Let $ \overline{(U,E)} $ and $ \overline{(V,E')} $ be classes of fuzzy soft sets over $ U $ and $ V $ with attributes from $ E $ and $ E' $
	respectively. Let $ p:U\longrightarrow V $ and  $ q:E\longrightarrow E' $ be two mappings. Then  $ f=(p,q):\overline{(U,E)}\longrightarrow\overline{(V,E')} $ is called a fuzzy soft mappings from $ \overline{(U,E)} $ to $ \overline{(V,E')}. $\\
	If $ p $ and $ q $ is injective then the fuzzy soft mapping $f=(p,q)  $ is said to be injective. \\
	If $ p $ and $ q $ is surjective then the fuzzy soft mapping $f=(p,q)  $ is said to be surjective. \\
	If $ p $ and $ q $ is constant then the fuzzy soft mapping $f=(p,q)  $ is said to be constant. \\
\end{defn}
\begin{defn} \cite{varol} Let $(U,E, \tau_1)$ and  $(U,E, \tau_2)$ be two fuzzy soft topological spaces.
	\begin{enumerate}
		\item[(i)] A fuzzy soft mapping $ f=(p,q):\overline{(U,E, \tau_1)}\longrightarrow\overline{(U,E, \tau_2)} $ is called fuzzy soft continuous if $ f^{-1} (G_B)\tilde{\in}\tau_1, \forall G_B\tilde{\in}\tau_2. $
				\item[(ii)]  A fuzzy soft mapping $ f=(p,q):\overline{(U,E, \tau_1)}\longrightarrow\overline{(U,E, \tau_2)} $ is called fuzzy soft open if $ f(F_A)\tilde{\in}\tau_2, \forall F_A\tilde{\in}\tau_1. $
	\end{enumerate}
\end{defn}
\begin{defn} \cite{akdag} Let $ (U, E, \tau_1) $ and $ (U, E, \tau_2) $ be two soft topological space. Then a soft multifunction $ f:\overline{(U,E, \tau_1)}\longrightarrow\overline{(U,E, \tau_2)}  $ is said to be ;
	\begin{enumerate}
		\item[(i)] soft upper semi continuous at a soft point $ e_i(F_A)\tilde{\in}(U,E) $if for every soft open set $ G_B \tilde{\in}(V,E)$ such that $f( e_i(F_A))\tilde{\subseteq}G_B, $ there exists a soft semi open neighborhood $ H_A $ of $ e_i(F_A) $ such that  $f( e_i(H_A))\tilde{\subseteq}G_B, \forall f( e_i(H_A)) \tilde{\in}H_A  $
				\item[(ii)] soft lower semi continuous at a soft point $ e_i(F_A)\tilde{\in}(U,E) $if for every soft open set $ G_B \tilde{\in}(V,E)$ such that $f( e_i(F_A))\tilde{\cap}G_B \tilde{\neq}\tilde{\phi}, $ there exists a soft semi open neighborhood $ H_A $ of $ e_i(F_A) $ such that  $f( e_i(H_A))\tilde{\cap}G_B \tilde{\neq}\tilde{\phi} , \forall f( e_i(H_A)) \tilde{\in}H_A  $
		\item[(iii)] soft upper(lower) semi continuous if $ f $ has this property at every soft point of $ (U, E) .$
		
	\end{enumerate}
\end{defn}

\section{Fuzzy soft numbers}

\begin{defn}
Let 	$F_A=\{ F(e_{i})=(h_{t}, \mu_{F(e_i)}(h_t));h_t\tilde{\in}U; t=1,2,...m;  i=1,2,...,n\} $ be a fuzzy soft set in $(U,E).$ Now convert object sets $ h_{t} $ are integers namely $ h_{t}=i, i=1,2,...,n.$ Then $ F_A $ is called convex fuzzy soft set if and only if membership function $ \mu_{F(e_i)} $ satisfies following conditions:
\begin{enumerate}
	\item[(i)] 	each parameter $ e_i $ of $ F_A,$ 	$\mu_{F(e_i)} (\lambda.h_{1}+(1-\lambda).h_{2})\geq \min\{\mu_{F(e_i)}(h_1),\mu_{F(e_i)}(h_2)\}\\  \textrm{where~}  \lambda\tilde{\in}[0,1].$ and $h_1, h_2\tilde{\in} \mathbb{R}.$
	\item[(ii)] $\mu_{\cap_iF(e_i)} (\lambda.h_{1}+(1-\lambda).h_{2})\geq \min\{\mu_{\cap_iF(e_i)}(h_1),\mu_{\cap_iF(e_i)}(h_2)\}  \textrm{~where~}  \lambda\tilde{\in}[0,1]. $ and $ h_1, h_2\tilde{\in} \mathbb{R}. $
\end{enumerate}
Otherwise it is non-convex fuzzy soft set.
\end{defn}
\begin{problem}
	Let \begin{align*}P_A=\{P(e_1)=\{(h_1, 0.1),(h_2, 0.6), (h_3, 1.0), (h_4, 0.8), (h_5, 0.2)\}\\	P(e_2)=\{(h_1, 0.3),(h_2, 0.9), (h_3, 1.0), (h_4, 0.7), (h_5, 0.2)\}\}.\end{align*}
	Therefore $ P_A $ be a convex fuzzy soft set.
	\end{problem}

\begin{thm}
If $ F_A $ and $ G_A $ are convex fuzzy soft sets then $ F_A \tilde{\cap} G_A $ is convex fuzzy soft set.
\end{thm}
\begin{proof}
Let  $ F_A $ and $ G_A $ are convex fuzzy soft sets in $ (U,E). $ Let $ h_1, h_2 \tilde{\in}U $ and $ e_1, e_2 \tilde{\in} E.$
Now convert objects $ h_1, h_2 $ are integers namely 1,2. Therefore
\begin{align*} &\mu_{F(e_1\tilde{\cap}e_2)} (\lambda.h_{1}+(1-\lambda).h_{2})\geq \min\{\mu_{F(e_1\tilde{\cap}e_2)}(h_1),\mu_{F(e_1\tilde{\cap}e_2)}(h_2)\}\\
&\text{and}\\
&\mu_{G(e_1\tilde{\cap}e_2)} (\lambda.h_{1}+(1-\lambda).h_{2})\geq \min\{\mu_{G(e_1\tilde{\cap}e_2)}(h_1),\mu_{G(e_1\tilde{\cap}e_2)}(h_2)\},\end{align*}
  where  $\lambda\tilde{\in}[0,1]$  and  $h_1, h_2\tilde{\in} \mathbb{R}.$\\
Now
\begin{align*}
&\mu_{F(e_1\tilde{\cap}e_2)\tilde{\cap}G(e_1\tilde{\cap}e_2)} (\lambda.h_{1}+(1-\lambda)h_{2})\\
&=\mu_{F(e_1\tilde{\cap}e_2)} (\lambda.h_{1}+(1-\lambda).h_{2}) \tilde{\cap}\mu_{G(e_1\tilde{\cap}e_2)} (\lambda.h_{1}+(1-\lambda).h_{2})\\
&\geq \min\{\mu_{F(e_1\tilde{\cap}e_2)}(h_1),\mu_{F(e_1\tilde{\cap}e_2)}(h_2)\}\tilde{\cap} \min\{\mu_{G(e_1\tilde{\cap}e_2)}(h_1),\mu_{G(e_1\tilde{\cap}e_2)}(h_2)\} \\
&\geq \min\{\mu_{F(e_1\tilde{\cap}e_2)\tilde{\cap}G(e_1\tilde{\cap}e_2)}(h_1),\mu_{F(e_1\tilde{\cap}e_2)\tilde{\cap}G(e_1\tilde{\cap}e_2)}(h_2)\}
 \text{~where~}  \lambda\tilde{\in}[0,1] \text{~and~}  h_1, h_2\tilde{\in} \mathbb{R}.\end{align*}
 Hence proved.
\end{proof}
\begin{thm}
If $ F_A $ and $ G_A $ are convex fuzzy soft sets and $ F_A \tilde{\subseteq} G_A $  then $ F_A \tilde{\cup}G_A $ and $ F_A \tilde{\cap} G_A $  are convex fuzzy soft set.
\end{thm}
\begin{proof}
Obvious.
\end{proof}
\begin{thm}
The union of any family of convex fuzzy soft sets is not necessarily a convex fuzzy soft set.
\end{thm}
\begin{proof}
The proof is straightforward.
\end{proof}
\begin{defn}
	Let 	$F_A=\{ F(e_{i})=(h_{t}, \mu_{F(e_i)}(h_t)); h_t\tilde{\in}U;  t=1,2,...m; i=1,2,...,n\} $ be a fuzzy soft set in $(U,E).$ Then $ F_A $ is called concave fuzzy soft set if and only if membership function $ \mu_{F(e_i)} $ satisfies following conditions:
\begin{enumerate}
\item[(i)] 	each parameter $ e_i $ of $ F_A,$\\
		$\mu_{F(e_i)} (\lambda.h_{1}+(1-\lambda).h_{2})\leq \max\{\mu_{F(e_i)}(h_1),\mu_{F(e_i)}(h_2)\}$  ~where~  $\lambda\tilde{\in}[0,1] $ ~and~ $ h_1, h_2\tilde{\in} \mathbb{R}. $
\item[(ii)] $\mu_{\cap_iF(e_i)} (\lambda.h_{1}+(1-\lambda).h_{2})\leq \max\{\mu_{\cap_iF(e_i)}(h_1),\mu_{\cap_iF(e_i)}(h_2)\}$  ~where~  $\lambda\tilde{\in}[0,1]. $ and $ h_1, h_2\tilde{\in} \mathbb{R}.$		
\end{enumerate}
Otherwise it is non-concave fuzzy soft set.
\end{defn}
\begin{problem}
	Let \begin{align*}N_A=\{N(e_1)=\{(h_1, 0.9),(h_2, 0.4), (h_3, 0.0), (h_4, 0.2), (h_5, 0.8)\}\\
	N(e_2)=\{(h_1, 0.7),(h_2, 0.1), (h_3, 0.0), (h_4, 0.3), (h_5, 0.8)\}\}.\end{align*}  	Therefore $ N_A $ be a concave fuzzy soft set.
	\end{problem}
\begin{prop}
If $ F_A $ and $ G_A $ are concave fuzzy soft sets. $ F_A \tilde{\cap} G_A $  and $ F_A \tilde{\cup} G_A $  are concave fuzzy soft sets.
\end{prop}
\begin{proof}
	Let  $ F_A $ and $ G_A $ are concave fuzzy soft sets in $ (U,E). $ Let $ h_1, h_2 \tilde{\in}U $ and $ e_1, e_2 \tilde{\in} E.$
	Now convert objects $ h_1, h_2 $ are integers namely 1,2. Therefore
	\begin{align*} &\mu_{F(e_1\tilde{\cap}e_2)} (\lambda.h_{1}+(1-\lambda).h_{2})\leq \max\{\mu_{F(e_1\tilde{\cap}e_2)}(h_1),\mu_{F(e_1\tilde{\cap}e_2)}(h_2)\}\\
		&\text{and}\\
		&\mu_{G(e_1\tilde{\cap}e_2)} (\lambda.h_{1}+(1-\lambda).h_{2})\leq \max\{\mu_{G(e_1\tilde{\cap}e_2)}(h_1),\mu_{G(e_1\tilde{\cap}e_2)}(h_2)\},\end{align*}
	where  $\lambda\tilde{\in}[0,1]$  and  $h_1, h_2\tilde{\in} \mathbb{R}.$\\
	Now
	\begin{align*}
		&\mu_{F(e_1\tilde{\cap}e_2)\tilde{\cap}G(e_1\tilde{\cap}e_2)} (\lambda.h_{1}+(1-\lambda)h_{2})\\
		&=\mu_{F(e_1\tilde{\cap}e_2)} (\lambda.h_{1}+(1-\lambda).h_{2}) \tilde{\cap}\mu_{G(e_1\tilde{\cap}e_2)} (\lambda.h_{1}+(1-\lambda).h_{2})\\
		&\leq \max\{\mu_{F(e_1\tilde{\cap}e_2)}(h_1),\mu_{F(e_1\tilde{\cap}e_2)}(h_2)\}\tilde{\cap} \max\{\mu_{G(e_1\tilde{\cap}e_2)}(h_1),\mu_{G(e_1\tilde{\cap}e_2)}(h_2)\} \\
		&\leq \max\{\mu_{F(e_1\tilde{\cap}e_2)\tilde{\cap}G(e_1\tilde{\cap}e_2)}(h_1),\mu_{F(e_1\tilde{\cap}e_2)\tilde{\cap}G(e_1\tilde{\cap}e_2)}(h_2)\}
		\text{~where~}  \lambda\tilde{\in}[0,1] \text{~and~}  h_1, h_2\tilde{\in} \mathbb{R}.\end{align*}
	Hence $ F_A \tilde{\cap} G_A $ be a concave fuzzy soft sets.\\
	Similarly we prove that  $ F_A \tilde{\cup} G_A $ be a concave fuzzy soft set.
\end{proof}
\begin{prop}\label{prop39}
Let $ F_A $ be a convex fuzzy soft set then $ F^C_A $ be a concave fuzzy soft sets.
\end{prop}
\begin{proof}
Let  $ F_A $ be a convex fuzzy soft set. Then
\begin{enumerate}
	\item[(i)] 	each parameter $ e_i $ of $ F_A,$ 	$\mu_{F(e_i)} (\lambda.h_{1}+(1-\lambda).h_{2})\geq \min\{\mu_{F(e_i)}(h_1),\mu_{F(e_i)}(h_2)\}\\  \textrm{where~}  \lambda\tilde{\in}[0,1].$ and $h_1, h_2\tilde{\in} \mathbb{R}.$
	\item[(ii)] $\mu_{\cap_iF(e_i)} (\lambda.h_{1}+(1-\lambda).h_{2})\geq \min\{\mu_{\cap_iF(e_i)}(h_1),\mu_{\cap_iF(e_i)}(h_2)\}$\\
where~  $\lambda\tilde{\in}[0,1] $ ~and~ $ h_1, h_2\tilde{\in} \mathbb{R}. $
\end{enumerate}
Now,
(i) 	Each parameter $ e_i $ of $ F^C_A,$ 	\begin{align*} &\mu_{F^C(e_i)} (\lambda.h_{1}+(1-\lambda).h_{2})\\
&=1-\mu_{F(e_i)} (\lambda.h_{1}+(1-\lambda).h_{2})\\
&\leq 1-\min\{\mu_{F(e_i)}(h_1),\mu_{F(e_i)}(h_2)\}\\
&\leq \max\{1-\mu_{F(e_i)}(h_1),1-\mu_{F(e_i)}(h_2)\}\\ &\leq \max\{\mu_{F^C(e_i)}(h_1),\mu_{F^C(e_i)}(h_2)\}, \end{align*}
 where~ $\lambda\tilde{\in}[0,1].$ and $h_1, h_2\tilde{\in} \mathbb{R}.$\\
		(ii)	Similarly we prove that \begin{align*}
	 \mu_{\cap_iF^C(e_i)} (\lambda.h_{1}+(1-\lambda).h_{2})\leq \max\{\mu_{\cap_iF^C(e_i)}(h_1),\mu_{\cap_iF^C(e_i)}(h_2)\} \end{align*}
Hence $F^C_A$ be a concave fuzzy soft sets.
\end{proof}
\begin{prop}
If $ F_A $ and $ G_A $ are  convex and concave fuzzy soft set, respectively and $ F_A \tilde{\subseteq}G_A .$ Then $ F_A \tilde{\cup}G_A $  and $ F_A \tilde{\cap}G_A $ are concave and convex  fuzzy soft set respectively.
\end{prop}
\begin{proof}
Since $ F_A \tilde{\subseteq}G_A ,$ therefore we have
 \[ F_A \tilde{\cup}G_A=G_A,\]
  which is concave fuzzy soft set. Then
 $ F_A \tilde{\cap}G_A$
  is a convex fuzzy soft set.
\end{proof}
\begin{defn}
A fuzzy soft set 	$F_A=\{ F(e_{i})=(h_{t}, \mu_{F(e_i)}(h_t)); h_t\tilde{\in}U; t=1,2,...m;  i=1,2,...,n\} $ is called a normalized fuzzy soft set if it satisfied following two conditions:
\begin{enumerate}
	\item[(i)] there is at least one point $h_t\tilde{\in}U $ with $\mu_{F(e_i)}(h_t)=1  $ for each $ e_i. $
	\item[(ii)] there is at least one point $h_t\tilde{\in}U $ with $\mu_{F(\cap_i(e_i))}(h_t)=1  $ ~for~ $\cap_i(e_i). $
\end{enumerate}
Otherwise it is non-normalized.
\end{defn}
\begin{problem}
	Let \begin{align*}K_A=\{K(e_1)=\{(h_1, 0.2), (h_2, 1.0),(h_3, 0.3)\},\\K(e_2)=\{(h_1, 0.1), (h_2, 1.0),(h_3, 0.2)\}\}.\end{align*}  Therefore $ K_A $ is a normalized fuzzy soft set.
\end{problem}
\begin{defn}
A fuzzy soft set	$F_A=\{ F(e_{i})=(h_{t}, \mu_{F(e_i)}(h_t)); h_t\tilde{\in}U; t=1,2,...m; i=1,2,...,n\} $ is a fuzzy soft number if its membership functions $\mu_{F(e_i)} $ is
\begin{enumerate}
	\item[(i)] fuzzy soft convex;
	\item[(ii)] fuzzy soft normalized;
	\item[(iii)] fuzzy soft upper semi-continuous.
	\item[(iv)] $cl\{h_t;\mu_{F(e_i)}(h_t)>0 \} $ is fuzzy soft  compact.
\end{enumerate}
\end{defn}
\begin{prop}
	Fuzzy soft numbers always normalized at same object of each parameters.
\end{prop}
\begin{proof}
Suppose $ F_A $ be a fuzzy soft number normalized at two objects $ h_2 $ and $ h_3 $ of parameters $ e_1 $ and $ e_2 $ respectively. Therefore
\[\mu_{F(e_1\tilde{\cap}e_2)}(h_2)\neq 1  \mbox{~and~} \mu_{F(e_1\tilde{\cap}e_2)}(h_3)\neq 1.\]
Therefore $ F_A $ is not a fuzzy soft number. \\
Hence Fuzzy soft numbers always normalized at same object of each parameters.
\end{proof}
\begin{prop}
Complement of fuzzy soft numbers is a concave fuzzy soft sets.
\end{prop}
\begin{proof}
Follows from the Proposition \ref{prop39}.
\end{proof}

\begin{defn}
Let $F_A=\{ F(e_{i})=(h_{t}, \mu_{F(e_i)}(h_t)); h_t\tilde{\in}U; t=1,2,...m; i=1,2,...,n\} $ and   $~G_A=\{ G(e_{i})=(h_{t}, \mu_{G(e_i)}(h_t)); h_t\tilde{\in}U; t=1,2,...m; i=1,2,...,n\} $ are two fuzzy soft numbers. Here we denote $ \mu_F $ and $ \mu_G $ are grade membership of objects of fuzzy soft sets $ F_A $ and $ G_A $ respectively.  Then four arithmetic operations for fuzzy soft numbers are defined as follows:\\
For $t=1,2,...,m$
\begin{enumerate}
\item[(i)] $F_A\tilde{+}G_A=\{ F(e_{i})\tilde{+} G(e_{i})\}=\{(h_{t},  \mu_F+\mu_G-\mu_F.\mu_G)\} .$
\item[(ii)] $F_A\tilde{-}G_A=\{ F(e_{i})\tilde{-} G(e_{i})\}=\{(h_{t}, \mu_F.\mu_G)\} .$
\item[(iii)] $F_A\tilde{\times}G_A=\{ F(e_{i})\tilde{\times} G(e_{i})\}=\left\{\left(h_{t},\frac{ \mu_F.\mu_G}{\max(\mu_F,\mu_G)}\right)\right\} .$
\item[(iv)] $F_A\tilde{\div}G_A=\{ F(e_{i})\tilde{\div}  G(e_{i})\}=\left\{\left(h_{t},\frac{ \mu_F}{\max(\mu_F,\mu_G)}\right)\right\} .$
\end{enumerate}
\end{defn}
\begin{problem}\label{prob317}
	Let \begin{align*}F_A=\{F(e_1)=\{(h_1, 0.0), (h_2, 0.6),(h_3, 1.0), (h_4, 0.8), (h_5, 0.1)\},\\F(e_2)=\{(h_1, 0.1), (h_2, 0.7),(h_3, 1.0), (h_4, 0.8), (h_5, 0.0)\}\}\end{align*} and
	 \begin{align*}G_A=\{G(e_1)=\{(h_1, 0.1), (h_2, 0.8),(h_3, 1.0), (h_4, 0.6), (h_5, 0.0)\},\\ G(e_2)=\{(h_1, 0.3), (h_2, 0.9),(h_3, 1.0), (h_4, 0.8), (h_5, 0.2)\}\}\end{align*} are two fuzzy soft numbers. Then
\begin{enumerate}
\item[(i)]$F_A\tilde{+}G_A= \{e_1=\{(h_1, 0.10), (h_2, 0.92),(h_3, 1.00), (h_4, 0.92), (h_5, 0.10)\},\\$  $e_2=\{(h_1, 0.37), (h_2, 0.97),(h_3, 1.00), (h_4, 0.96), (h_5, 0.10)\}\} $
\item[(ii)]$F_A\tilde{-}G_A= \{e_1=\{(h_1, 0.00), (h_2, 0.48),(h_3, 1.00), (h_4, 0.48), (h_5, 0.00)\},\\$  $e_2=\{(h_1, 0.03), (h_2, 0.63),(h_3, 1.00), (h_4, 0.64), (h_5, 0.00)\}\} $
\item[(iii)]$F_A\tilde{\times}G_A= \{e_1=\{(h_1, 0.00), (h_2, 0.60),(h_3, 1.00), (h_4, 0.60), (h_5, 0.00)\},\\$  $e_2=\{(h_1, 0.10), (h_2, 0.70),(h_3, 1.00), (h_4, 0.80), (h_5, 0.00)\}\} $
\item[(iv)]$F_A\tilde{\div}G_A= \{e_1=\{(h_1, 0.00), (h_2, 0.75),(h_3, 1.00), (h_4, 1.00), (h_5, 1.00)\},\\$  $e_2=\{(h_1, 0.34), (h_2, 0.78),(h_3, 1.00), (h_4, 1.00), (h_5, 1.00)\}\}.$
\end{enumerate}
\end{problem}
\begin{prop}
	If membership value of same objects of both soft numbers are zero, then multiplication and division operations of that soft numbers are undefined.
\end{prop}
\begin{proof}
Follows from definition.	
\end{proof}

\begin{prop}
Let $ F_A $ and $ G_A $ are two fuzzy soft numbers, then $F_A\tilde{+}G_A,$ $ F_A\tilde{-}G_A$ and $F_A\tilde{\times}G_A    $ are also  fuzzy soft numbers.
\end{prop}
\begin{proof}
Let $F_A=\{ F(e_{i})=(h_{t}, \mu_{F(e_i)}(h_t))\} $ and $G_A=\{ G(e_{i})=(h_{t}, \mu_{G(e_i)}(h_t))\} $ are two fuzzy soft numbers. Therefore
\[F_A\tilde{+}G_A=\{ F(e_{i})\tilde{+} G(e_{i})\}=\{(h_{t},  \mu_F+\mu_G-\mu_F.\mu_G)\} \]
~and~
 \[F_A\tilde{-}G_A=\{ F(e_{i})\tilde{-} G(e_{i})\}=\{(h_{t}, \mu_F.\mu_G)\} .\]
 Since membership functions $ \mu_F $ and $\mu_G  $ are convex and normalized. Therefore $\mu_F+\mu_G-\mu_F.\mu_G  $ and $\mu_F.\mu_G  $ are convex and normalized. Also $\mu_F+\mu_G-\mu_F.\mu_G  $ and $\mu_F.\mu_G  $ are upper semi-continuous and closure of that values$ (>0) $ are fuzzy soft compact. Hence $F_A\tilde{+}G_A, F_A\tilde{-}G_A$ are fuzzy soft numbers. \\
 Now Consider $F_A\tilde{\times}G_A.$
 Since
 \[F_A\tilde{\times}G_A=\{ F(e_{i})\tilde{\times} G(e_{i})\}=\left\{\left(h_{t},\frac{ \mu_F.\mu_G}{\max(\mu_F,\mu_G)}\right)\right\}.\]
 Case 1: If $ F_A \tilde{\subseteq}G_A, $then $F_A\tilde{\times}G_A= \{(h_{t}, \mu_F)\} =F_A.$\\
  Case 2: If $ G_A \tilde{\subseteq}F_A, $then $F_A\tilde{\times}G_A= \{(h_{t}, \mu_G)\} =G_A.$\\
 Case 3: If $ F_A =G_A, $then $F_A\tilde{\times}G_A=G_A=F_A.$\\
 Hence $F_A\tilde{\times}G_A   $ is fuzzy soft number.
\end{proof}
\begin{prop}
	Let $ F_A $ and $ G_A $ are two fuzzy soft numbers, then
\begin{enumerate}	
\item[(i)] $F_A\tilde{+}G_A= G_A\tilde{+}F_A.$	
\item[(ii)]$F_A\tilde{\times}G_A= G_A\tilde{\times}F_A.$		
\item[(iii)]$F_A\tilde{+}\tilde{\phi}=F_A.$	
\item[(iv)]$F_A\tilde{\times}\tilde{E}= F_A.$
\item[(v)]$F_A\tilde{\times}\tilde{\phi}= \tilde{\phi}.$
\item[(vi)]$\tilde{\phi}\tilde{\div}F_A= \tilde{\phi}.$
\item[(vii)]$F_A\tilde{\div}\tilde{\phi}= \tilde{E}.$
\item[(viii)]$F_A\tilde{\div}F_A= \tilde{E}.$
\end{enumerate}	
\end{prop}
\begin{proof}
Obvious. 	
\end{proof}
\begin{prop}
Let $ F_A $, $ G_A $ and $ H_A $ are three fuzzy soft numbers, then
\begin{enumerate}	
	\item[(i)] $(F_A\tilde{+}G_A)\tilde{+}H_A= F_A\tilde{+}(G_A\tilde{+}H_A).$	
	\item[(ii)]$(F_A\tilde{\times}G_A)\tilde{\times}H_A=F_A\tilde{\times}( G_A\tilde{\times}H_A).$		
	\item[(iii)]$F_A\tilde{+}(G_A\tilde{\cup}H_A)= (F_A\tilde{+}G_A)\tilde{\cup}(F_A\tilde{+}H_A).$	
	\item[(iv)]$F_A\tilde{+}(G_A\tilde{\cap}H_A)= (F_A\tilde{+}G_A)\tilde{\cap}(F_A\tilde{+}H_A).$
	\item[(v)]$F_A\tilde{-}(G_A\tilde{\cup}H_A)= (F_A\tilde{-}G_A)\tilde{\cap}(F_A\tilde{-}H_A).$
	\item[(vi)]$F_A\tilde{-}(G_A\tilde{\cap}H_A)= (F_A\tilde{-}G_A)\tilde{\cup}(F_A\tilde{-}H_A).$
	\item[(vii)]$(G_A\tilde{\cup}H_A)\tilde{-}F_A= (G_A\tilde{-}F_A)\tilde{\cup}(H_A\tilde{-}F_A).$	
	\item[(viii)]$(G_A\tilde{\cap}H_A)\tilde{-}F_A= (G_A\tilde{-}F_A)\tilde{\cap}(H_A\tilde{-}F_A).$	
	\item[(ix)]$F_A\tilde{\times}(G_A\tilde{\cup}H_A)= (F_A\tilde{\times}G_A)\tilde{\cup}(F_A\tilde{\times}H_A).$
	\item[(x)]$F_A\tilde{\times}(G_A\tilde{\cap}H_A)= (F_A\tilde{\times}G_A)\tilde{\cap}(F_A\tilde{\times}H_A).$
	\item[(xi)]$F_A\tilde{\div}(G_A\tilde{\cup}H_A)= (F_A\tilde{\div}G_A)\tilde{\cap}(F_A\tilde{\div}H_A).$
	\item[(xii)]$F_A\tilde{\div}(G_A\tilde{\cap}H_A)= (F_A\tilde{\div}G_A)\tilde{\cup}(F_A\tilde{\div}H_A).$
	\item[(xiii)]$(G_A\tilde{\cup}H_A)\tilde{\div}F_A= (G_A\tilde{\div}F_A)\tilde{\cup}(H_A\tilde{\div}F_A).$
	\item[(xiv)]$(G_A\tilde{\cap}H_A)\tilde{\div}F_A= (G_A\tilde{\div}F_A)\tilde{\cap}(H_A\tilde{\div}F_A).$
	
\end{enumerate}	
\end{prop}
\begin{proof}
Obvious.	
\end{proof}

\section{Hausdorff distance between two fuzzy soft numbers}
\begin{defn}
Let $ F_A $ and $ G_A $ are two fuzzy soft numbers. The Hausdorff distance between $ F_A $ and $ G_A $ is defined as
\[\tilde{d}(F_A, G_A)= \max\{\mid \mu_{F(\cap_i(e_i))}(h_t)-\mu_{G(\cap_i(e_i))}(h_t)\mid \},\]
where $ i=1,2,...,n; t=1,2,....,m. $
\end{defn}
\begin{problem}
From example \ref{prob317}
\begin{align*}
\tilde{d}(F_A, G_A)&= \max\{\mid 0.0-0.1\mid, \mid 0.6-0.8\mid, \mid 1.0-1.0\mid, \mid 0.8-0.6\mid, \mid 0.0-0.0\mid \}\\
&= \max\{0.1, 0.2, 0.0, 0.2, 0.0\}=0.2
\end{align*}
\end{problem}

\begin{defn}
	Consider a fuzzy soft point $ e_i(F_A); (i=1,2,..n )$ in a soft number  $ F_A $ and $ G_A $ be any fuzzy soft number. The Hausdorff distance between soft point and $ G_A $ is defined by
	\[\tilde{d}(e_i(F_A), G_A)= \max\{\mid \mu_{e_i(F_A)}(h_t)-\mu_{G(\cap_i(e_i))}(h_t)\mid \},\]
	where $ i=1,2,...,n; t=1,2,....,m. $
	\end{defn}
\begin{problem}
	From example \ref{prob317}
	and consider a soft point $ e_1(F_A) $ in a soft number $ F_A $ as
\[ e_1(F_A)=\{(h_1, 0.0),(h_2, 0.6), (h_3, 1.0), (h_4, 0.8), (h_5, 0.1) \} .\] Therefore
\begin{align*}
	\tilde{d}(e_1(F_A), F_A)&= \max\{\mid 0.0-0.0\mid, \mid 0.6-0.6\mid, \mid 1.0-1.0\mid, \mid 0.8-0.8\mid, \mid 0.1-0.0\mid \}\\
	&= \max\{0.0, 0.0, 0.0, 0.0, 0.1\}=0.1\end{align*}
and
\begin{align*}	
	\tilde{d}(e_1(F_A), G_A)&= \max\{\mid 0.0-0.1\mid, \mid 0.6-0.8\mid, \mid 1.0-1.0\mid, \mid 0.8-0.6\mid, \mid 0.1-0.0\mid \}\\
		&= \max\{0.1, 0.2, 0.0, 0.2, 0.1\}=0.2.\end{align*}
\end{problem}
\begin{prop}
Distance between soft number and complement of soft number always one.
\end{prop}
\begin{proof}
Consider the fuzzy soft number $ F_A. $ Therefore
\begin{align*}
\tilde{d}(F_A, F^C_A)&= \max\{\mid 0.0-0.9\mid, \mid 0.6-0.3\mid, \mid 1.0-0.0\mid, \mid 0.8-0.2\mid, \mid 0.0-0.9\mid \}\\
&= \max\{0.9, 0.3, 1.0, 0.6, 0.9\}=1.0.\end{align*}
\end{proof}
\begin{thm}
If $ L_A, M_A , N_A $ are three fuzzy soft numbers and $ L_A \tilde{\subseteq}M_A\tilde{\subseteq}H_A . $ Then
\begin{enumerate}
	\item[(i)] $\tilde{d}(L_A, M_A)\tilde{\leq}\tilde{d}(L_A, H_A).$
	\item[(ii)] $\tilde{d}(M_A, H_A)\tilde{\leq}\tilde{d}(L_A, H_A).$
\end{enumerate}
\end{thm}
\begin{proof}
Consider
 \begin{align*}L_A=\{L(e_1)=\{(h_1, 0.2),(h_2, 0.8), (h_3, 1.0), (h_4, 0.7), (h_5, 0.1)\},\\
L(e_2)=\{(h_1, 0.3),(h_2, 0.9), (h_3, 1.0), (h_4, 0.7), (h_5, 0.2)\}\};\end{align*}
\begin{align*}M_A=\{M(e_1)=\{(h_1, 0.1),(h_2, 0.8), (h_3, 1.0), (h_4, 0.6), (h_5, 0.0)\}\\
M(e_2)=\{(h_1, 0.1),(h_2, 0.7), (h_3, 1.0), (h_4, 0.7), (h_5, 0.0)\}\}\end{align*}
and
\begin{align*}H_A=\{H(e_1)=\{(h_1, 0.1),(h_2, 0.7), (h_3, 1.0), (h_4, 0.6), (h_5, 0.0)\}\\
H(e_2)=\{(h_1, 0.0),(h_2, 0.7), (h_3, 1.0), (h_4, 0.7), (h_5, 0.0)\}\}.\end{align*}
 Therefore
\[\tilde{d}(L_A, M_A)= 0.1;  \tilde{d}(L_A, H_A)= 0.2;  \tilde{d}(M_A, H_A)= 0.1. \]
Hence  $\tilde{d}(L_A, M_A)\tilde{\leq}\tilde{d}(L_A, H_A)$ and
$\tilde{d}(M_A, H_A)\tilde{\leq}\tilde{d}(L_A, H_A).$
\end{proof}
\begin{defn}
Let $ F_A $ be a fuzzy soft number over $ (U,E) .$ A mapping $ \tilde{d}: (U,E)\times(U,E)\longrightarrow[0,1] $ is said to be a fuzzy soft metric of soft numbers over $(U,E),$ if $ \tilde{d} $ satisfies the following conditions:
\begin{enumerate}
	\item[(i)] $ \tilde{d}(F_A, G_A)\tilde{\geq}0, \forall F_A, G_A \tilde{\in}(U, E),$
	\item[(ii)] $ \tilde{d}(F_A, G_A)=0 \Longleftrightarrow F_A=G_A, $
	\item[(iii)]$  \tilde{d}(F_A, G_A)= \tilde{d}(G_A, F_A) ,$
	\item[(iv)]$ \tilde{d}(F_A, G_A)\tilde{\leq}\tilde{d}(F_A, G_A)+\tilde{d}(G_A, H_A) .$
\end{enumerate}
Then the triple  $(U,E,\tilde{d}) $ is called a fuzzy soft metric space. 

\end{defn}
\begin{defn}
Let $ (U, E, \tilde{d}) $ be a fuzzy soft metric space. Then the diameter of soft number $ F_A $ is denoted by
\begin{align*}
&\tilde{\delta}(F_A)=\max\{\tilde{d}(e_i(F_A), e_j(F_A))\}; \text{~for~ each~} e_i,e_j \tilde{\in}F_A. \end{align*}
\end{defn}
\begin{prop}
Diameter of each fuzzy soft point number is always zero.
\end{prop}
\begin{proof}
We know that distance between each fuzzy soft point is zero. Therefore maximum distance between each soft point is also zero. \\
Hence proved the result.
\end{proof}
\begin{thm}
If $ F_A \tilde{\subseteq}G_A $ of $ (U, E) ,$ then $  \tilde{\delta}(F_A)\tilde{\leq}\tilde{\delta}(G_A) .$
\end{thm}
\begin{proof}
Since	$ F_A \tilde{\subseteq}G_A, $ therefore we have
\[e_i(F_A) \tilde{\subseteq}e_i(G_A), e_j(F_A) \tilde{\subseteq}e_j(G_A), \forall e_i, e_j\tilde{\in}F_A, G_A.\]
	Now
	 \begin{align*}
	&\tilde{\delta}(F_A)=\max\{\tilde{d}(e_i(F_A), e_j(F_A))\} \tilde{\leq}\max\{\tilde{d}(e_i(G_A), e_j(G_A))\}\tilde{\leq}\tilde{d}(G_A) \end{align*}
	for each $e_i,e_j \tilde{\in}F_A, G_A.$
	This completes the proof.
\end{proof}
\begin{thm}
Let $ F_A, G_A \tilde{\in}(U, E) $ and $ F_A \tilde{\cap} G_A \tilde{\neq}\tilde{\phi} .$  Then $\tilde{\delta}(F_A \tilde{\cup} G_A )\tilde{\leq}\tilde{\delta}(F_A) \tilde{+}\tilde{\delta}(G_A)  . $
\end{thm}
\begin{proof}
Follows from definition.
\end{proof}

\begin{defn}
	Let $ (U, E, \tilde{d}) $ be a fuzzy soft metric space and $ F_A \tilde{\in}(U, E). $ Then for any $ r\tilde{\in}(0,1), $ the set
	\[ S_r(F_A)=\{G_A\tilde{\in}(U,E): \tilde{d}(F_A, G_A)\tilde{<}r\} \]
	 is called a fuzzy soft open sphere  of radius $ 'r' $ centered at $ F_A. $
\end{defn}

\begin{defn}
	Let $ (U, E, \tilde{d}) $ be a fuzzy soft metric space  and $ F_A \tilde{\in}(U, E).$ Then for any $ r\tilde{\in}(0,1), $ the set
	\[ S_r[F_A]=\{G_A\tilde{\in}(U,E): \tilde{d}(F_A, G_A)\tilde{\leq} r\} \]
	is called  fuzzy soft closed sphere  of radius $ 'r' $ centered at $ F_A. $
\end{defn}
\begin{defn}
	Let $ (U, E, \tilde{d}) $ be a fuzzy soft metric space  and $ F_A \tilde{\in}(U, E). $ A sub class $ (V^{'}, F) $ of $ (U, E) $ is
	called a fuzzy soft neighborhood of a soft number $ F_A \tilde{\in}(U, E), $ we denoted as $ N^ {(V^{'}, F)}_{F_A} ,$ if there exists an open sphere of fuzzy soft number  $S_r(F_A)  $ center at $ F_A $ and contained in $(V^{'}, F) .  $
	i.e.  $ S_r(F_A) \tilde{\subseteq} N^ {(V^{'}, F)}_{F_A} , $ for some $ r\tilde{\in}(0,1). $
	\end{defn}
\begin{defn}
A sub class $ (V^{'}, F) $ of a fuzzy soft metric space $ (U, E, \tilde{d}) $ is said to be open in  $ (U, E,\tilde{d}),$ if  $ (V^{'}, F) $ is a fuzzy  soft neighborhood of each of its soft numbers.
i.e.if for each $ F_A\tilde{\in}(V^{'}, F),  $ there is an  $ r\tilde{\in}(0,1) $ such that
 $ S_r(F_A) \tilde{\subseteq}(V^{'}, F). $
\end{defn}
\begin{defn}
Let $ \{(F_A)_n\} $ be a sequence of fuzzy soft numbers in a fuzzy soft metric space $ (U, E, \tilde{d}). $ The sequence $ \{(F_A)_n\} $ is said to converge in $ (U, E, \tilde{d}) $ if there a fuzzy soft number $ F^{'}_A $ in $ (U,E) $ such that
$\tilde{d}((F_A)_n, F^{'}_A)\tilde{\rightarrow} 0$ as $ n \tilde{\rightarrow }\infty.  $\\
That is for every $ \varepsilon\tilde{>}0 $ there exists a positive integer $m$ such that  \[\tilde{d}((F_A)_n, F^{'}_A)\tilde{\leq}\varepsilon,~ \forall n\geq m.\]
It is denoted as $ \lim_{n \rightarrow \infty} (F_A)_n=F^{'}_A.$
\end{defn}
\begin{defn}
Let $ \{(F_A)_n\} $ be a sequence of fuzzy soft numbers in a fuzzy soft metric space $ (U, E, \tilde{d}). $ Then $ \{(F_A)_n\} $ is said to be bounded, there exists a positive number $ \beta\tilde{\in}(0,1] $ such that
$\tilde{d}((F_A)_n,(F_A)_m)\tilde{\leq}\beta, \forall n, m \tilde{\in} \beta. $
\end{defn}

\begin{defn}
	Let $ \{(F_A)_n\} $ be a sequence of fuzzy soft numbers in a fuzzy soft metric space $ (U, E, \tilde{d}). $ Then $ \{(F_A)_n\} $ is said to be Cauchy sequence of fuzzy soft numbers if for a positive number $ \varepsilon\tilde{>}0 $ there exists a positive integer $ \beta ,$ such that
	\[\tilde{d}((F_A)_n,(F_A)_m)\tilde{\leq}\varepsilon, \forall n, m \tilde{\geq}\beta ~i.e.~ \tilde{d}((F_A)_n,(F_A)_m)\tilde{\rightarrow} 0   \mbox{~as~} n, m \tilde{\rightarrow} \infty.  \]
	\end{defn}
\begin{thm}
Every convergent sequence in a fuzzy soft numbers  is Cauchy sequence and every Cauchy sequence of  fuzzy soft numbers is bounded.
\end{thm}
\begin{proof}
	Let $ \{(F_A)_n\} $ be a sequence of fuzzy soft numbers in a fuzzy soft metric space $ (U, E, \tilde{d}). $Let $ \{(F_A)_n\} $ converges to $ F^{'}_A .$For every $ \varepsilon\tilde{>}0 $ there exists a positive integer $\beta$ such that  \[\tilde{d}((F_A)_n, F^{'}_A)\tilde{\leq} \frac{\varepsilon}{2},~ \forall n\geq \beta.\]
	Then for all $ m,n\geq \beta,$ we have
		\[\tilde{d}((F_A)_n,(F_A)_m)\tilde{\leq}\tilde{d}((F_A)_n,F^{'}_A)+\tilde{d}(F^{'}_A,(F_A)_m)\tilde{<}\frac{\varepsilon}{2}+ \frac{\varepsilon}{2}=\varepsilon \]
		Hence $ \{(F_A)_n\} $ is a Cauchy sequence. \\
Let $ \{(F_A)_n\} $ is a Cauchy sequence of fuzzy soft number. Therefore
\[\tilde{d}((F_A)_n,(F_A)_m)\tilde{\leq}\varepsilon \tilde{\in} (0, 1]\]
Hence $ \{(F_A)_n\} $ is bounded.
	\end{proof}
\begin{defn}
A fuzzy soft numbers metric space $ (U, E, \tilde{d}) $ is called complete fuzzy soft metric space if every Cauchy sequence of fuzzy soft numbers in $ (U, E) $ converges to some fuzzy soft number in  $ (U, E). $ 
\end{defn}
\begin{defn}
Let $ \overline{(U, E)} $ and $ \overline{(V, E^{'})} $ be classes of fuzzy soft numbers over $ U $ and $ V $ with attributes from $ E $ and $ E^{'},$ respectively. Let $ p: U \longrightarrow V $ uniformly one-one onto and $ q: E \longrightarrow E^{'} $ be any mapping. Then $ f=(p, q): (U, E) \longrightarrow (V, E^{'}) $ is a fuzzy soft numbers mapping.
\end{defn}
\begin{problem}\label{prob422}
Let $ U=\{h_1, h_2, h_3, h_4, h_5\} $	and $ V=\{k_1, k_2, k_3, k_4, k_5\}, E=\{e_1, e_2, e_3\}, E'=\{e_1', e_2'\} $ and $ \overline{(U,E)},\overline{(V,E')} $ classes of fuzzy soft numbers. Let $ p(h_1)=k_1, p(h_2)=k_2, p(h_3)=k_3, p(h_4)=k_4, p(h_5)=k_5 $ and $ q(e_1)=e_2', q(e_2)=e_1', q(e_3)=e_2' . $ Let us  consider a fuzzy soft number $ H_A $ in $\overline{(U,E)} $ as
\begin{align*}
H_A=\{H(e_1)=\{(h_1, 0.0), (h_2, 0.6),(h_3, 1.0), (h_4, 0.8), (h_5, 0.1)\},\\
 H(e_2)=\{(h_1, 0.1), (h_2, 0.7),(h_3, 1.0), (h_4, 0.8), (h_5, 0.0)\},\\
 H(e_3)=\{(h_1, 0.3), (h_2, 0.9),(h_3, 1.0), (h_4, 0.7), (h_5, 0.2)\}\}.\end{align*}
Then the fuzzy soft number image of $ H_A $ under  $ f=(p,q):\overline{(U,E)}\longrightarrow\overline{(V,E')} $ is obtained as
\begin{align*}  f(H_A)(e_1')(k_1)&=\bigcup_{\alpha\tilde{\in}q^-1(e_1')\cap A,s\tilde{\in}p^-1(k_1)} (\alpha)\mu_s\\
&=\bigcup_{\alpha\tilde{\in}\{e_2\},s\tilde{\in}\{h_1\}} (\alpha)\mu_s\\
&=(e_2)\mu_{h_1}=\{0.1\}\end{align*}
\begin{align*}
 f(H_A)(e_1')(k_2)&=\bigcup_{\alpha\tilde{\in}q^-1(e_1')\cap A,s\tilde{\in}p^-1(k_2)} (\alpha)\mu_s\\
&=\bigcup_{\alpha\tilde{\in}\{e_2\},s\tilde{\in}\{h_2\}} (\alpha)\mu_s\\
&=(e_2)\mu_{h_2}=\{0.7\}.\end{align*}
	By similar calculations we get
\begin{align*}
 f(H_A)=\{e_1'=\{(k_1, 0.1), (k_2, 0.7),(k_3, 1.0), (k_4, 0.8), (k_5, 0.0)\},\\e_2'=\{(k_1, 0.3), (k_2, 0.9),(k_3, 1.0), (k_4, 0.8), (k_5, 0.2)\}\}.\end{align*}
Again consider a fuzzy soft number $ H_B' $ in $ \overline{(V,E')} $ as
\begin{align*}
H_B'=\{e_1'=\{(k_1, 0.2), (k_2, 0.8),(k_3, 1.0), (k_4, 0.7), (k_5, 0.1)\},\\e_2'=\{(k_1, 0.1), (k_2, 0.6),(k_3, 1.0), (k_4, 0.7), (k_5, 0.0)\}.\end{align*}
	Therefore
\[f^{-1}(H_B')(e_1)(h_1)=(q(e_1))\mu_{p(h_1)}=(e_2')\mu_{K_1}=\{0.1\}.\]
By similar calculations, we get
\begin{align*}
f^{-1}(H_B')=\{e_1=\{(h_1, 0.1), (h_2, 0.6),(h_3, 1.0), (h_4, 0.7), (h_5, 0.0)\}\\e_2=\{(h_1, 0.2), (h_2, 0.8),(h_3, 1.0), (h_4, 0.8), (h_5, 0.0)\}\\e_3=\{(h_1, 0.1), (h_2, 0.6),(h_3, 1.0), (h_4, 0.7), (h_5, 0.0)\}\}.\end{align*}
\end{problem}
\begin{prop}
Let $ \overline{(U, E)} $ and $ \overline{(V, E^{'})} $ be classes of fuzzy soft numbers.  Let $ p: U \longrightarrow V $ be not uniformly one-one onto and $ q: E \longrightarrow E^{'}$ be mapping. Then $ f=(p, q): (U, E) \longrightarrow (V, E^{'}) $ is not a fuzzy soft numbers mapping.
\end{prop}
\begin{proof}
From example \ref{prob422}, here we consider  $ p(h_1)=k_3, p(h_2)=k_5, p(h_3)=k_1, p(h_4)=k_2, p(h_5)=k_4.$
Then the image of $ H_A $ under  $ f=(p,q):\overline{(U,E)}\longrightarrow\overline{(V,E')} $ is obtained as
\begin{align*} f(H_A)(e_1')(k_1)&=\bigcup_{\alpha\tilde{\in}q^-1(e_1')\cap A,s\tilde{\in}p^-1(k_1)} (\alpha)\mu_s\\
&=\bigcup_{\alpha\tilde{\in}\{e_2\},s\tilde{\in}\{h_3\}} (\alpha)\mu_s\\
&=(e_2)\mu_{h_3}=\{1.0\}.\end{align*}
Similar we get
\begin{align*}
 f(H_A)=\{e_1'=\{(k_1, 1.0), (k_2, 0.8),(k_3, 0.1), (k_4, 0.0), (k_5, 0.7)\},\\e_2'=\{(k_1, 1.0), (k_2, 0.8),(k_3, 0.3), (k_4, 0.2), (k_5, 0.9)\}\end{align*}	
which is not a fuzzy soft numbers. This complete the proof.
\end{proof}
\begin{defn}
A fuzzy soft number mapping $ f=(p, q): (U, E) \longrightarrow (V, E^{'}) $ is said to be a one-one and onto if $ q: E\longrightarrow E^{'} $ be a one-one onto.
\end{defn}
\begin{thm}\label{thm435}
Suppose a fuzzy soft number mapping $ f=(p, q): (U, E) \longrightarrow (V, E^{'}) $ is one-one onto  of two fuzzy soft  metric spaces
$ (U, E, \tilde{d_1}) $ and  $ (V, E^{'}, \tilde{d_2}). $ If $ F_A $ and $ G_A $ are two fuzzy soft numbers of $ (U, E), $ then   $\tilde{d_1}(F_A, G_A)=\tilde{d_2}(f(F_A), f(G_A)).$
\end{thm}
\begin{proof}
Since $ f $ is one-one onto. Therefore $ p $ and $ q $ are also one-one onto. Therefore
 \[\mu_{F(\cap_i(e_i))}(h_t)=\mu_{f(F)(\cap_i(e_i))}(h_t)\] and  \[\mu_{G(\cap_i(e_i))}(h_t)=\mu_{f(G)(\cap_i(e_i))}(h_t)\] for all $F_A, G_A \tilde{\in}(U,E).$ \\
 Hence,  $\tilde{d_1}(F_A, G_A)=\tilde{d_2}(f(F_A), f(G_A)).$
\end{proof}
\begin{defn}
Let  $ (U, E, \tilde{d_1}) $ and  $ (V, E^{'}, \tilde{d_2}) $ be any two fuzzy soft  metric spaces. A fuzzy soft number function (or mapping) $ f: (U, E) \longrightarrow (V, E^{'}) $ is said to be fuzzy soft continuous at $ F^{'}_A $ of $ (U, E), $ if for given $ \varepsilon \tilde{>}0 $ there exists a $ \delta \tilde{>} 0, $ such that 	$\tilde{d_2}(f(F_A), f(F^{'}_A))\tilde{\leq}\varepsilon, $ whenever $\tilde{d_1}(F_A, F^{'}_A)\tilde{\leq}\delta. $ \\

	i.e., for each fuzzy soft open sphere  $ S_\epsilon(f(F^{'}_A )) $ centered at $f(F^{'}_A ) $ there is a  fuzzy soft open sphere  $ S_\delta(F^{'}_A ) $ centered at $ F^{'}_A  $ such that
	\[f(S_\delta(F^{'}_A ) )\tilde{\subseteq}S_\varepsilon(f(F^{'}_A )). \]
\end{defn}
\begin{defn}
A fuzzy soft number function $ f: (U, E, \tilde{d_1}) \longrightarrow (V, E^{'}, \tilde{d_2}) $ is said to be fuzzy soft  continuous, if it is fuzzy soft  continuous at each fuzzy soft number of $ (U, E). $
\end{defn}
\begin{prop}
	Every inverse image of soft number function is fuzzy soft  continuous.
\end{prop}
\begin{proof}
	Obvious.		
\end{proof}

\begin{thm}
Let  $ (U, E, \tilde{d_1}) $ and  $ (V, E^{'}, \tilde{d_2}) $ be any two fuzzy soft metric spaces and $ f: (U, E) \longrightarrow (V, E^{'}) $ be fuzzy soft continuous. Then for every soft number sequence $ {(F_A)_n} $ converges to $ F^{'}_A $ we have
$ \lim_{n\tilde{\rightarrow} \infty}f((F_A)_n)=f( F^{'}_A) $
i.e $(F_A)_n \tilde{\rightarrow}F^{'}_A\Rightarrow f((F_A)_n)\tilde{\rightarrow}f( F^{'}_A).$
\end{thm}
\begin{proof}
Since $ f $ is a fuzzy soft number function. Therefore $ f $ is a fuzzy soft  continuous. Let $ f $ be a fuzzy soft  continuous at $ F^{'}_A.$ Therefore for any  given $ \varepsilon \tilde{>}0 $ there exists a $ \delta \tilde{>} 0, $ such that
\begin{equation}\label{eqa41}
\tilde{d_2}(f(X_A), f(F^{'}_A))\tilde{\leq}\varepsilon,  \mbox{~whenever~}
\tilde{d_1}(X_A, F^{'}_A)\tilde{\leq}\delta. \end{equation}
Let us suppose that $ {(F_A)_n} $ be a soft number sequence in  $ (U, E,\tilde{d_1}),$ such that
\[\lim_{n\tilde{\rightarrow} \infty}f((F_A)_n)=f( F^{'}_A).\]
Since  $ \lim_{n\tilde{\rightarrow} \infty}(F_A)_n= F^{'}_A, $ therefore there exists a positive integer $ m $ such that
\[\tilde{d_1}((F_A)_n, F^{'}_A)\tilde{\leq}\delta, \forall n\geq m. \]
Now from (\ref{eqa41}), we get
\begin{align*}
&\tilde{d_2}(f((F_A)_n), f(F^{'}_A))\tilde{\leq}\varepsilon, \forall n\geq m.\end{align*}
This implies that
$ \lim_{n\tilde{\rightarrow} \infty}f((F_A)_n)=f( F^{'}_A) .$
\end{proof}
\begin{defn}
	Let  $ (U, E, \tilde{d_1}) $ and  $ (V, E^{'}, \tilde{d_2}) $ be any two fuzzy soft  metric spaces. A fuzzy soft number function $ f: (U, E) \longrightarrow (V, E^{'}) $ is said to be fuzzy soft uniformly  continuous  if for each $ \varepsilon \tilde{>}0 $ there exists a $ \delta \tilde{>} 0, $ such that
	\[\tilde{d_2}(f(X_A), f(Y_A))\tilde{\leq}\varepsilon, \mbox{~whenever~}
	\tilde{d_1}(X_A, Y_A)\tilde{\leq}\delta,  \forall X_A, Y_A \tilde{\in} (U,E).\]
\end{defn}
\begin{prop}
Every fuzzy soft number function is uniformly continuous.
\end{prop}
\begin{proof}
Obvious.
\end{proof}
\begin{thm}
Every fuzzy soft number function  $ f: (U, E) \longrightarrow (V, E^{'}) $ which is uniformly continuous on $(U, E)  $ is necessarily soft continuous on $(U, E) . $
\end{thm}

\begin{proof}
	Obvious.
\end{proof}
\begin{thm}
The continuous image of a fuzzy soft Cauchy sequence is again a fuzzy soft Cauchy sequence.
\end{thm}
\begin{proof}
	Let  $ (U, E, \tilde{d_1}) $ and  $ (V, E^{'}, \tilde{d_2}) $ be any two fuzzy soft metric spaces and  $ f: (U, E) \longrightarrow (V, E^{'}) $ be fuzzy soft number function. \\
	We have every soft number function is uniformly continuous.	Therefore $ f $ is uniformly continuous. \\
	Let $ \{(F_A)_n\} $ be a fuzzy soft number Cauchy sequence in $ (U,E) $ and given $ \varepsilon \tilde{>}0. $ Then $ f $ being a fuzzy soft  uniformly continuous, there exists a $ \delta \tilde{>}0 $ such that
		\begin{equation}\label{eqa42}
		\tilde{d_2}(f((F_A)_m), f((Y_A)_n))\tilde{\leq}\varepsilon \mbox{~whenever~}
		\tilde{d_1}((F_A)_m, (F_A)_n)\tilde{\leq}\delta. \end{equation}
		Since $ \{(F_A)_n\} $ is fuzzy soft Cauchy sequence, corresponding to this  $ \delta \tilde{>}0 $ there exists a positive integer $ n_0 $ such that
		\begin{equation}\label{eqa43}
		\tilde{d_1}((F_A)_m, (F_A)_n)\tilde{\leq}\delta  \mbox{~~for~all~} m, n \tilde{\geq}n_0. \end{equation}
		From (\ref{eqa42}) and (\ref{eqa43}) we have
			\[\tilde{d_2}(f((F_A)_m), f((Y_A)_n))\tilde{\leq}\varepsilon \mbox{~~for~all~} m, n \tilde{\geq}n_0.  \]
			Hence
	 $ \{(F_A)_n\} $  is a soft  Cauchy sequence in $(V, E^{'}).  $
		 \end{proof}
\begin{defn}
	Let  $ (U, E, \tilde{d_1}) $ and  $ (V, E^{'}, \tilde{d_2}) $ be any two fuzzy soft metric spaces. A fuzzy soft number function $ f=(p, q): (U, E) \longrightarrow (V, E^{'}) $ is said to be fuzzy soft homeomorphism if $ q: E\longrightarrow  E^{'} $ is one-one and onto.
\end{defn}
\begin{thm}
	Every fuzzy soft number one one onto function is fuzzy soft homeomorphism.
\end{thm}
\begin{proof}
Suppose $ f=(p, q): (U, E) \longrightarrow (V, E^{'}) $ is one-one onto function. Therefore $ p: U \longrightarrow V$ uniformly one-one onto and $ q: E\longrightarrow E^{'} $ is one-one onto.\\
Hence $ f $ is fuzzy soft homeomorphism. 	
\end{proof}
\begin{defn}
	Let  $ (U, E, \tilde{d_1}) $ and  $ (V, E^{'}, \tilde{d_2}) $ be any two fuzzy soft  metric spaces. A fuzzy soft number function $ f: (U, E) \longrightarrow (V, E^{'}) $ is called a fuzzy soft  isometry if
	\[ \tilde{d_1}(X_A, Y_A)=\tilde{d_2}(f(X_A), f(Y_A)) , \forall X_A, Y_A \tilde{\in}(U, E).\]
\end{defn}
\begin{thm}
	Every one one onto fuzzy soft number function is a  fuzzy soft  isometry function.
\end{thm}

\begin{proof}
Follows from Theorem \ref{thm435}.
\end{proof}
\begin{thm}\label{thm438}
	Every fuzzy soft  isometry function is not a one one onto fuzzy soft number function.
\end{thm}

\begin{proof}
Consider the example \ref{prob422} and let
\begin{align*} Q_A=\{Q(e_1)=\{(h_1, 0.3), (h_2, 0.9),(h_3, 1.0), (h_4, 0.7), (h_5, 0.2)\},\\ Q(e_2)=\{(h_1, 0.1), (h_2, 0.7),(h_3, 1.0), (h_4, 0.8), (h_5, 0.0)\},\\ Q(e_3)=\{(h_1, 0.0), (h_2, 0.8),(h_3, 1.0), (h_4, 0.7), (h_5, 0.0)\} \}.\end{align*}
Therefore
\begin{align*}  f(Q_A)=\{e_1'=\{(k_1, 0.1), (k_2, 0.7),(k_3, 1.0), (k_4, 0.8), (k_5, 0.0)\},\\e_2'=\{(k_1, 0.3), (k_2, 0.9),(k_3, 1.0), (k_4, 0.7), (k_5, 0.2)\}\}.\end{align*}
Now
\[ \tilde{d_1}(H_A, Q_A)= 0.1=\tilde{d_2}(f(H_A), f(Q_A)).\]  So $ f $  is fuzzy soft  isometry function but it is not a one one onto fuzzy soft number function.
\end{proof}
\begin{thm}
	Every fuzzy soft  isometry function is not a fuzzy soft  homeomorphism.
\end{thm}
\begin{proof}
From Theorem \ref{thm438}, 	every fuzzy soft  isometry function is not a one one onto fuzzy soft number function. Therefore from definition of homeomorphism,every fuzzy soft  isometry function is not a fuzzy soft  homeomorphism.
\end{proof}
\noindent {\bf Funding:} No funds were received.\\
{\bf Availability of data and materials:} 
No data were used to support this study.\\
{\bf Competing interests:} The authors declare no conflict of interest.\\
{\bf Authors’ Contributions:} 
All authors contributed equally to the manuscript and typed,
read, and approved the final manuscript.


\end{document}